\documentclass[11pt]{article}
\usepackage{amssymb}
\usepackage{amsmath}
\numberwithin{equation}{section}
\usepackage{amsthm}
\usepackage{amsfonts}
\usepackage{mathrsfs}
\usepackage{makeidx}
\usepackage{multicol}
\newtheorem{theorem}{Theorem}[section]

\newtheorem{conjecture}[theorem]{Conjecture}
\newtheorem{problem}[theorem]{Problem}

\begin{document}
\title{\bf A Connection between the Riemann Hypothesis and Uniqueness of the Riemann zeta function }

\author{ Pei-Chu Hu  and Bao Qin Li}
\date{}
\maketitle

\vspace{3mm}

Department of Mathematics, Shandong University, Jinan 250100,

Shandong, P. R. China   E-mail: pchu@sdu.edu.cn

\vspace{3mm}

Department of Mathematics and Statistics, Florida International

University, Miami, FL 33199 USA    E-mail: libaoqin@fiu.edu

\vspace{1cm}

\begin{abstract}
In this paper, we give a connection between the Riemann hypothesis and uniqueness of the Riemann zeta function and an analogue for L-functions.
\end{abstract}

\pagenumbering{arabic}

\begin{figure}[b]
\rule[-2.5truemm]{5cm}{0.1truemm}\\[2mm] {\footnotesize
Mathematics Subject Classification 2000 (MSC2000): 11M36, 30D35. 
\\ 
\\}
\end{figure}

\section{Introduction}
The Riemann $\zeta$ function is defined by the Dirichlet series
\begin{equation}\label{zeta-Dirichletser}
\zeta(s)=\sum_{n=1}^{\infty} \frac{1}{n^{s}},\quad s=\sigma+it
\end{equation}
for ${\rm Re}(s)>1$, which is absolutely convergent, and admits an analytical continuation as a meromorphic function in the complex plane $\mathbb{C}$ of order $1$, which has only a simple pole at $s=1$ with residue equal to $1$. It satisfies the following Riemann functional equation:
\begin{equation}\label{r-equation}
\zeta(1-s)=2(2\pi)^{-s}\cos\left(\frac{\pi
s}{2}\right)\Gamma(s)\zeta(s),
\end{equation}
where $\Gamma$ is the Euler gamma function
$$
\Gamma(z) =\int_0^{\infty}t^{z-1}e^{-t}dt, \ \   \ \ \mbox{Re} z
>0,
$$
analytically continued as a meromorphic function in $\mathbb{C}$ of order $1$ without any zeros and with simple poles at $s=0$, $-1$, $-2$, $\cdots$.

The allied function
\begin{equation}\label{xi}
\xi(s)=\frac{s}{2}(s-1)\pi^{-\frac{s}{2}}\Gamma\left(\frac{s}{2}\right)\zeta(s)
\end{equation}
is an entire function of order equal to $1$ satisfying the functional equation
\begin{equation}\label{xi-func-eq}
\xi(1-s)=\xi(s)
\end{equation}
(see e.g. \cite{T2}, p.16 and p.29).
\vspace{3mm}

It is easy to see that $\zeta(s)$ has no zeros for ${\rm Re}(s)>1$ and, by the functional equation, the only zeros of $\zeta(s)$ in the domain ${\rm Re}(s)<0$ are the poles of $\Gamma(s/2)$. These are called the
{\it trivial zeros}\index{trivial zero} of $\zeta(s)$. Other zeros, called {\it nontrivial zeros}, lie in the {\it critical strip}\index{critical strip} $0\leq{\rm Re}(s)\leq 1$ (actually lie in the open strip $0<{\rm Re}(s)<1$). It is a well-known theorem of G. H. Hardy that there are an infinity of
zeros on ${\rm Re}(s)=\frac{1}{2}$. The famous, as  yet unproven, Riemann hypothesis states as follows:

\begin{conjecture}[Riemann Hypothesis]
The nontrivial zeros of $\zeta(s)$ lie on the line ${\rm
Re}(s)=\frac{1}{2}$.
\end{conjecture}

The uniqueness problem for the Riemann zeta function (more generally, for L-functions, see below) is to study how the Riemann zeta function $\zeta$ (or an L-function) is uniquely determined by its zeros or by its $a$-values, i.e., the zeros of $\zeta(s)-a$, where $a$ is a complex value. Uniqueness problems have extensively been studied in the value distribution theory of meromorphic functions in terms of shared values (see e.g. the monographs \cite{HLY} and \cite{YY}), in which two meromorphic functions $f$ and $g$ are called to share a value $a$ if $Z(f-a)=Z(g-a)$, where $Z(F)$ denotes the zero set of $F$ (counting or not counting multiplicities, depending on the questions under consideration). The problem for the Riemann zeta function and L-functions has recently been studied in various settings (see e.g. \cite{CY}, \cite{GHK}, \cite{Ki}, \cite{KL}, \cite{Li1}, \cite{Li2}, \cite{St}, to list a few). In particular, in \cite{KL} and \cite{Li1}, the problem was considered by relaxing the set equality $Z(f-a)=Z(g-a)$ to the set inclusion $Z(f-a)\subseteq Z(g-a)$ for uniqueness of L-functions, which will be seen to be crucial in \S 2. Roughly speaking, two L-functions satisfying the same functional equation are identically equal if they have sufficiently many common zeros (see \cite{KL} and \cite{Li1} for the details and related results as well as references), which gives a uniqueness theorem for solutions of the Riemann functional equation or, more generally, Riemann type functional equations (cf. \S 2 and see [1], [4], [9], etc. for studies of solutions of the Riemann functional equation).

In the present paper, we will discover a connection (an equivalence) between the Riemann Hypothesis and the above mentioned uniqueness problem and then an analogue for L-functions, which, as a consequence, also implies a simply stated necessary and sufficient condition for the Riemann Hypothesis to hold in terms of the limit of an allied function as $\sigma\to +\infty$. This connection does not seem to have been observed before. The results it has brought out in this paper are of a neat and best possible form. Given the fact that uniqueness problems have been studied extensively for meromorphic functions and various techniques have been developed over the years, it would be profitable to further explore this approach with the connection in mind.

\section{Results}

Let $\rho_n$ be the nontrivial zeros of $\zeta$ in the critical strip
$0\leq{\rm Re}(s)\leq 1$. It follows that
\begin{equation*}
\zeta(\rho_n)=\zeta(1-\rho_n)=\zeta(\bar{\rho}_n)
=\zeta(1-\bar{\rho}_n)=0
\end{equation*}
from the functional equation and the identity
$
\zeta(\bar{s})=\overline{\zeta(s)},
$
that is, $\bar{\rho}_n$, $1-\rho_n$, $1-\bar{\rho}_n$ are zeros of
$\zeta(s)$, too. In other words, nontrivial zeros of $\zeta(s)$
are distributed symmetrically with respect to the real axis and to
the critical line ${\rm Re}(s)=\frac{1}{2}$.

Now, let $s_\nu$ be the zeros of $\zeta$ on the half-line ${\rm Re}(s)=\frac{1}{2}$, ${\rm
Im}(s)>0$. Assume that $\rho_n$, $s_\nu$ are ordered with respect
to increasing absolute values of their imaginary parts.

We would like to define a meromorphic function that captures the key features of the Riemann zeta function $\zeta$ with, however, all its zeros in the critical strip being located exactly at those nontrivial zeros of $\zeta$ on the critical line ${\rm Re}(s)=\frac{1}{2}$. If this function is defined ``ideally" and turns out to be identically equal to $\zeta$ (It is here where the uniqueness problem arises, cf. below), then the Riemann hypothesis must follow by the distribution of the zeros of the constructed function.

To realize this goal, we first construct an entire function which plays the role of $\xi$ with, however, zeros at $s_\nu$, which are, as mentioned above, distributed symmetrically with respect to the critical line. We define
\begin{equation}\label{h-symmetry}
h(s)=\frac{1}{2}\prod_{\nu=1}^\infty\left(1-\frac{s-s^2}{|s_\nu|^2}\right).
\end{equation}
This function $h$ possesses the following properties, which are important in serving our purposes:

\medskip
(a) $h$ is an entire function of order $\le 1$;

(b) The general factor $1-\frac{s-s^2}{|s_\nu|^2}$ in the infinite product has exactly the zeros at $s_v$ and $\overline{s_\nu}$ and thus the zeros of the function $h$ are exactly $s_v$ and $\overline{s_\nu}$ (symmetrically with respect to the critical line),  $v=1, 2, \cdots$;

(c) $\lim\limits_{s\to 1}h(s)=\frac{1}{2}$;

(d) $h$ satisfies the same equation (\ref{xi-func-eq}) as $\xi$ does, i.e.,
\begin{equation}\label{h(s)}
h(1-s)=h(s).
\end{equation}

To see (a), it is clear from the definition of $\xi$ in (\ref{xi}) that all the zeros of $\xi$ lie in the critical strip and they are zeros of $\zeta$, i.e., $\rho_n$. Recall that $\xi$ is of order $1$. It follows from Jensen's formula that
\begin{equation*}
n(r, \{ s_{\nu} \})\le n(r, \{\rho_n\})\le Kr^{1+\epsilon}
\end{equation*}
for any $\epsilon>0$, where $K>0$ is a constant and $n(r, \{\rho_n\})$ (resp. $n(r, \{ s_{\nu} \})$) denotes the number of the points $\rho_n, n=1, 2,\cdots$ (resp. $s_\nu, \nu=1, 2, \cdots$) lying in the disc $|s|\le r$ (see e.g. \cite{T1}, p.249). Thus for $|s|\ge 1$,
\begin{eqnarray*}
& &\log |h(s)|\le  \sum_{\nu=1}^{\infty}\log\left(1+|\frac{2s^2}{s_{\nu}^2}|\right)-\log 2 \nonumber \\
& &=\int_0^{\infty}\log\left(1+\frac{2|s^2|}{r^2}\right)dn(r, \{ s_{\nu} \})-\log 2  \nonumber\\
& &=4|s|^2\int_0^{\infty}\frac{n(r, \{ s_{\nu} \})}{ r(r^2+2|s|^2)}dr -\log 2 \nonumber\\
& &\le 4|s|^2\big(\int_0^{|s|}\frac{Kr^{1+\epsilon}}{ 2r|s|^2}dr+\int_{|s|}^{\infty}\frac{Kr^{1+\epsilon}}{ r^3}dr\bigr) -\log 2 \nonumber\\
& &=\frac{2K}{1+\epsilon}|s|^{1+\epsilon}+\frac{4K}{1-\epsilon}|s|^{1+\epsilon}-\log 2,
\end{eqnarray*}
which implies that $h$ is an entire function of order $\le 1$. This shows (a). The property (b) is immediate by the fact that $(s-s_v)(s-\overline{s_\nu})=s^2-s+|s^2_{\nu}|$ since ${\rm Re}(s_\nu)=\frac{1}{2}$. The properties (c) and (d) are also immediate, directly from the expression of $h$ in (\ref{h-symmetry}).

Further, we define a meromorphic function $\eta(s)$ using the same expression as that for $\zeta$ in (\ref{xi}) (with the role of $\xi$ there being replaced by $h$),
\begin{equation}\label{eta}
h(s)=\frac{s}{2}(s-1)\pi^{-\frac{s}{2}}\Gamma\left(\frac{s}{2}\right)\eta(s).
\end{equation}
Replacing $s$ by $1-s$ yields that
\begin{equation*}
h(1-s)=\frac{s}{2}(s-1)\pi^{-\frac{1-s}{2}}\Gamma\left(\frac{1-s}{2}\right)\eta(1-s),
\end{equation*}
which implies, in view of (\ref{h(s)}) and (\ref{eta}), that
\begin{eqnarray*}
& &\eta(1-s)=\frac{h(1-s)}{ \frac{s}{2}(s-1)\pi^{-\frac{1-s}{2}}\Gamma\left(\frac{1-s}{2}\right)}\\
& &=\frac{\frac{s}{2}(s-1)\pi^{-\frac{s}{2}}\Gamma\left(\frac{s}{2}\right)\eta(s)}{ \frac{s}{2}(s-1)\pi^{-\frac{1-s}{2}}\Gamma\left(\frac{1-s}{2}\right)}\\
& &=\frac{\pi^{-s+\frac{1}{2}}\Gamma(\frac{s}{2})}{\Gamma(\frac{1-s}{2})}\eta(s)\\
& &=2(2\pi)^{-s}\cos\left(\frac{\pi s}{2}\right)\Gamma(s)\eta(s),
\end{eqnarray*}
by virtue of the identity $\frac{\pi^{-s+\frac{1}{2}}\Gamma(\frac{s}{2})}{\Gamma(\frac{1-s}{2})}=2(2\pi)^{-s}\cos\left(\frac{\pi s}{2}\right)\Gamma(s)$ (see e.g. \cite{T2}, p.16). That is, the function $\eta$ also satisfies the Riemann functional equation  (\ref{r-equation}) as $\zeta$ does:
\begin{equation}\label{eta1}
\eta(1-s)=2(2\pi)^{-s}\cos\left(\frac{\pi
s}{2}\right)\Gamma(s)\eta(s).
\end{equation}

We see from (\ref{eta}) that only zeros of $\eta(s)$ in the domain ${\rm Re}(s)<0$ are the poles of $\Gamma(s/2)$, which are the trivial zeros of $\zeta$. Other zeros of $\eta$ lie on the line ${\rm Re}(s)=\frac{1}{2}$ in view of the construction of $h(s)$ and $\eta(s)$ (see Property (b) of $h$). The point $s=1$ is the only pole of $\eta(s)$, which is a simple pole with residue
\begin{equation}\label{residue}
    {\rm Res}_{s=1}\eta(s)=\lim_{s\to 1}(s-1)\eta(s)=\lim_{s\to 1}\frac{2\pi^{\frac{s}{2}}h(s)}{s\Gamma\left(\frac{s}{2}\right)}=\frac{\sqrt{\pi}}{\Gamma\left(\frac{1}{2}\right)}=1,
\end{equation}
using (\ref{eta}) and Property (c) of $h$. It also follows from (\ref{eta}) that $(s-1)\eta$ is an entire function of order $\le 1$ in view of Property (a) of $h$.

The function $\eta$ possesses the characteristics we desire, as described above; in fact, we can now establish the following

\begin{theorem}\label{RH-eqthm}
The Riemann hypothesis is true if and only if $\zeta(s)\equiv \eta(s)$.
\end{theorem}

\begin{proof}
The sufficiency is clear since all the zeros of $\eta$ on ${\rm Re}(s)\ge 0$ lie on the line  ${\rm Re}(s)=\frac{1}{2}$ by the construction of $h(s)$ and $\eta(s)$.

For the necessity, if the Riemann hypothesis holds, then $\zeta$ and $\eta$ have the same zeros in the entire complex plan; thus we have that
$\eta(s)=e^{as+b}\zeta(s)$ for some complex numbers $a, b$, in view of the fact that $\zeta$ is of order $1$ and $\eta$ is of order $\le 1$. We deduce, by applying (\ref{eta}), (\ref{xi}), (\ref{h(s)}) and (\ref{xi-func-eq}), that
\begin{eqnarray*}
& &e^{as+b}=\frac{\eta(s)}{\zeta(s)}=\frac{h(s)}{\xi(s)}\\
& &=\frac{h(1-s)}{\xi(1-s)}=\frac{\eta(1-s)}{\zeta(1-s)}=e^{a(1-s)+b}.
\end{eqnarray*}
Thus, $e^{as+b}=e^{a(1-s)+b},$ which implies that $a=0$. Then $\eta(s)=e^{b}\zeta(s)$ and then $(s-1)\eta(s)=e^{b}(s-1)\zeta(s)$. Taking the limit $s\to 1$ and by (\ref{residue}) and the fact that $\zeta$ has residue $1$ at $s=1$ also, we obtain that $e^b=1$. This proves that $\eta(s)\equiv \zeta(s).$
\end{proof}

From Theorem 2.1, to prove the Riemann hypothesis we now only need to prove that $\zeta(s)\equiv\eta(s)$, from which the uniqueness problem arises. Note that the function $\eta$ is a meromorphic function in $\mathbb{C}$ of order $\le 1$ that satisfies the following important properties:

(i) $\eta$ and $\zeta$ satisfy the same functional equation;

(ii) the zero set of $\eta$ is a subset of the zero set of $\zeta$ (counting multiplicities), i.e., $Z(\eta)\subseteq Z(\zeta)$, where $Z(f)$ denotes the set of the zeros of $f$ with counting multiplicities.

\medskip
The property (i) means that $\eta$ is a solution of the Riemann functional equation, which is known to have different solutions with certain relations (see \cite{BC}, \cite{Ha}, \cite{Kn}, etc. for studies of solutions of the Riemann functional equation). Clearly we are seeking the conditions that force the solutions to become the unique one - the Riemann zeta function. This leads to the following uniqueness problem:

\begin{problem}[Uniqueness problem] Let $f$ be a meromorphic function (of order $\le 1$) in $\mathbb{C}$ such that

(i) $f$ and $\zeta$ satisfy the same functional equation;

(ii) $Z(f)\subseteq Z(\zeta)$.

Under what conditions are $f$ and $\zeta$ identically equal?
\end{problem}

This is the uniqueness problem considered in \cite{Li1} and then in \cite{KL} for two L-functions; but to serve our purpose here we now need to consider the uniqueness problem when one of the functions is a meromorphic function $f$ satisfying the above two conditions (i) and (ii) in Problem 2.2.

It is clear that if $f$ satisfies the above two conditions (i) and (ii), then for any nonzero constant $c$, $cf$ also satisfies the these two conditions. An obvious property of the Riemann zeta function (simply from its Dirichlet series) is that $\zeta$ tends to $1$ as $\sigma\to +\infty$. In order to have the uniqueness of $f$ and $\zeta$, $f$ must necessarily tend to $1$ as $\sigma\to +\infty$. Thus, this naturally becomes the condition we use, as given in the theorem below.

\begin{theorem}\label{zeta-thm}
Let $f(s)$ be a nonconstant meromorphic function in $\mathbb{C}$ of order $\le 1$ with $\lim\limits_{\sigma\to +\infty}f(s)=1$. Then  $f\equiv \zeta$ if and only if $f$ satisfies the Riemann functional equation and $Z(f)\subseteq Z(\zeta)$.
\end{theorem}

The theorem will be proved later and treated as a consequence of a more general result for L-functions (see Theorem 2.5 below). Since the function $\eta$ satisfies the conditions (i) and (ii) in Problem 2.2, Theorem 2.3 and Theorem 2.1 yield immediately the following theorem for the Riemann hypothesis, which is of a particularly neat and simple statement:

\begin{theorem}
The Riemann hypothesis is true if and only if $\lim\limits_{\sigma\to +\infty}\eta(s)=1$.

\end{theorem}

In fact, if $\lim\limits_{\sigma\to +\infty}\eta(s)=1$, then $\eta$ satisfies all the conditions of Theorem 2.3 and thus, $\eta(s)\equiv \zeta(s)$, which implies that the Riemann hypothesis is true by the sufficient condition of Theorem 2.1. Conversely, if the Riemann hypothesis holds, then by the necessary condition of Theorem 2.1, $\zeta(s)\equiv\eta(s)$, which then implies that $\lim\limits_{\sigma\to +\infty}\eta(s)=\lim\limits_{\sigma\to +\infty}\zeta(s)=1$.

\medskip
We are going to generalize Theorem 2.3 so that one of the functions in the theorem is a meromorphic function $f$ as described above and the other is a Dirichlet series in the extended Selberg class, which takes the Riemann zeta function as a special case, so the above approach can then be pushed over to L-functions (see Theorem 2.5 below). The result we present is more than what we need, which is inspired by and based on our earlier work \cite{Li1} and \cite{KL}, and which, as a uniqueness theorem, is of its own independent interest. The observation that one of the functions is not necessarily assumed to be an L-function is essential for the purpose of the connection as analyzed above.

Recall that the Selberg class of $L$-functions is the set of all Dirichlet series $L(s)=\sum_{n=1}^{\infty} {a(n)\over n^s}$ with $a(1)=1$, satisfying the following axioms (see \cite{Se}):\medskip

\noindent (i) (Dirichlet series) For $\sigma >1 $, $L(s)$ is an absolutely convergent Dirichlet series;
\hfil\break (ii) (Analytic
continuation) There is a non-negative integer $k$ such that
$(s-1)^kL(s)$ is an entire function of finite order; \hfil\break
(iii) (Functional equation) $L$ satisfies a functional equation of
type
$$\Lambda_L(s)=\omega\overline{\Lambda_L(1-{\bar
s})},$$ where $\Lambda_L(s)=L(s)Q^s\prod_{j=1}^K\Gamma(\lambda_j
s+\mu_j)$ with positive real numbers $Q, \lambda_j$, and complex
numbers $\mu_j, \omega$ with $\hbox{Re}\mu_j\geq 0$ and
$|\omega|=1$; \hfil\break (iv) (Ramanujan hypothesis) $a(n)\ll
n^{\varepsilon}$ for every $\varepsilon>0$;
\hfil\break(v) (Euler product) $\log L(s)=\sum_{n=1}^{\infty}{b(n)\over n^s}$, where
$b(n)=0$ unless $n$ is a positive power of a prime and $b(n)\ll
n^{\theta}$ for some $\theta<{1\over 2}$.
\medskip

The Selberg class includes the Riemann zeta-function $\zeta$ and
essentially all Dirichlet series where one might expect the
analogue of the Riemann hypothesis. In the uniqueness theorem given below, all $L$-functions are assumed to be in the extended Selberg class, i.e., Dirichlet series
$L(s)=\sum_{n=1}^{\infty} {a(n)\over n^s}$ with $a(1)=1$ satisfying only the axioms (i)-(iii). Thus, the result below particularly applies to $L$-functions in the Selberg class.

\medskip
\begin{theorem}\label{M-thm}
Let $f(s)$ be a nonconstant meromorphic function in $\mathbb{C}$ of order $\le 1$ with $\lim\limits_{\sigma\to +\infty}f(s)=1$ and $L$ an L-function. Then $f\equiv L$ if and only if $f$ and $L$ satisfy the same functional equation and
$Z^+(f)\setminus G\subseteq Z^+(L)$ for a set $G$ (counted with multiplicity) satisfying
that
\begin{equation}\label{G-notation}
\limsup\limits_{r\rightarrow \infty}\frac{n(r, G)}{r}<{\log 4\over \pi}.
\end{equation}
Furthermore, the inequality (\ref{G-notation}) is best possible.
\end{theorem}

On the above,  $n(r, G)$ denotes the number of points of $G$ (counting multiplicities) lying in the disc $|s|\leq r$. And, $Z^+(L)$ denotes the set of nontrivial zeros of $L$ counted with multiplicity. As usual, the trivial zeros of $L$ are those coming from the poles of the $\Gamma$ factors in the functional equation of the axiom (iii), and the other zeros are called nontrivial zeros. The set $Z^+(f)$ is defined in the same way using the same functional equation.

\medskip
In addition to the sharpness of (\ref{G-notation}), the conditions in Theorem 2.5 (and thus in Theorem 2.3) are tight and the result is best possible in the sense that the theorem breaks down if any of the conditions is dropped, as shown by the counterexamples in the Remark after the proof.

\medskip
\begin{proof} By the assumption on the set $G$, it is easy to check that the infinite product $\sum\limits_{\rho \in
G}\log\left(1-\frac{s^2}{\rho^2}\right)$ converges to an entire function in the complex plane (cf. (\ref{product}) below). Since $L$ satisfies the analytic continuation axiom (ii), $L$ has at most one pole at $s=1$. We can thus properly choose integers $m, n$ such that the auxiliary function
\begin{equation}\label{function}
F(s):=\left(s^2-1\right)^ms^n{L(s)-f(s)\over f(s)}{L(-s)-f(-s)\over f(-s)}\prod_{\rho\in
G}\left(1-{s^2\over \rho^2}\right)
\end{equation}
does not have a pole at $s=\pm 1$ and that $s=0$ is a zero of $F$ (we may then assume that $s=0$ is not in $G$). Since $f$ and $L$ satisfy the same functional equation, the function $f-L$ must satisfy the same functional equation. Thus, $f$ and  $f-L$ have the same trivial zeros that are located at the poles of the $\Gamma$ factors in the functional equation of the axiom (iii). These zeros do not produce any poles of $F$ due to cancelation. Other zeros of $f$ are canceled by those of $L-f$ in (\ref{function}). Any pole of $f$ clearly does not produce a pole of $F$ by the construction of $F$. Hence, $F$ is an entire function.

Choose
$0<D_1<D_2<1$ with $\limsup\limits_{r\rightarrow \infty}\frac{n(r, G)}{r}<D_1{\log 4\over \pi}.$
Then, there is a positive number $r_0>0$ such that ${n(r, G)\over r}<D_1{\log 4\over \pi}$ for $r\geq r_0$. We deduce that for large $|s|$,
\begin{eqnarray}\label{product}
& &\log\left|\prod_{\rho\in
G}\left(1-\frac{s^2}{\rho^2}\right)\right|\leq\sum_{\rho \in
G}\log\left(1+|\frac{s^2}{\rho^2}|\right)  \nonumber \\
& &=\int_0^{\infty}\log\left(1+\frac{|s^2|}{r^2}\right)dn(r, G)  \nonumber\\
& &=2|s|^2\int_0^{\infty}\frac{n(r, G)}{ r(r^2+|s|^2)}dr  \nonumber\\
& &\leq 2|s|^2\left\{\frac{1}{ |s|^2}\int_0^{r_0}\frac{n(r, G)}{ r}dr+ D_1\frac{\log 4}{\pi}\int_{r_0}^{\infty}\frac{1}{ r^2+|s|^2}dr  \right\} \nonumber \\
& &= 2|s|^2\left\{\frac{1}{ |s|^2}\int_0^{r_0}\frac{n(r, G)}{ r}dr+ D_1\frac{\log 4}{\pi}\frac{1}{|s|}(\frac{\pi}{2}-\arctan\frac{r_0}{|s|})\right\} \nonumber \\
& &\le 2\int_0^{r_0}\frac{n(r, G)}{r}dr+D_1|s|\log 4\leq D_2|s|\log 4.
\end{eqnarray}

Recall that $L(s)=\sum_{n=1}^{\infty}{a(n)\over n^s}$ with
$a(1)=1$ and the series converges absolutely as $\sigma>1$. It is elementary to check that $(\frac{n}{2})^{\sigma}\geq n^2$ for $n\ge 4$ and $\sigma\ge 4$. Thus as $\sigma\ge 4$, we have that
\begin{eqnarray*}
& &\sum_{n=4}^{\infty}|\frac{a(n)}{n^s}|\le \frac{1}{2^{\sigma}}\sum_{n=4}^{\infty}|\frac{a(n)}{(\frac{n}{2})^{\sigma}}| \\
& &\le \frac{1}{2^{\sigma}}\sum_{n=4}^{\infty}|\frac{a(n)}{n^2}|=\frac{C}{2^{\sigma}},
\end{eqnarray*}
where $C=\sum\limits_{n=4}^{\infty}|\frac{a(n)}{n^2}|<+\infty$. Hence,
$$|f(s)-L(s)|=|f(s)-1-\sum\limits_{n=2}^{\infty}\frac{a_n}{n^s}|\le |f(s)-1|+\frac{1}{2^{\sigma}}O(1)$$
and then for a fixed $\epsilon>0$ (to be specified later),
$$|\frac{f(s)-L(s)}{f(s)}|=(\epsilon+\frac{1}{2^\sigma})O(1)$$
for large $\sigma$, in view of the assumption that $f(s)\to 1$ as $\sigma\to +\infty$.

Dividing the functional equation of $f-L$ by the same functional equation satisfied by $f$ and $L$, we obtain that
\begin{equation*}
{L(s)-f(s)\over f(s)}={\overline{L(1-\overline{s})}-\overline{f(1-\overline{s})}\over \overline{f(1-\overline{s})}}.
\end{equation*}
We thus obtain that
$$\left|{L(s)-f(s)\over f(s)}\cdot{L(-s)-f(-s)\over f(-s)}\right|=(\epsilon+\frac{1}{2^{|\sigma|}})^2O(1)$$
as $\sigma\rightarrow \pm \infty$. By applying this estimate and the estimate (\ref{product}) to (\ref{function}), we have
that for a number $D_3$ with $D_2<D_3<1$,
\begin{eqnarray*}
& &\log |F(s)|\leq  D_3|s|\log 4+2\log (\epsilon+\frac{1}{2^{|\sigma|}})\\
& &=D_3|s|\log 4-|\sigma|\log 4+2\log (1+\epsilon 2^{|\sigma|})
\end{eqnarray*}
as $\sigma\rightarrow\pm \infty.$

Define $g(\epsilon)=\log (1+\epsilon 2^{|\sigma|})-\epsilon\log 2^{|\sigma|}$ for $\epsilon>0$. Then $g(0)=0$ and it is easy to check that $g'(\epsilon)<0$ for sufficiently large $|\sigma|$. Thus, as $\sigma\rightarrow\pm \infty,$ we have that
$$\log (1+\epsilon 2^{|\sigma|})\le \epsilon\log 2^{|\sigma|}.$$
We can now take $\epsilon$ such that $D:=D_3+\epsilon<1$. Then
\begin{eqnarray*}
& &\log |F(s)|\leq  D_3|s|\log 4-|\sigma|\log 4+\epsilon |\sigma|\log 4\\
& & \le (D|s|-|\sigma|)\log 4=|s|(D-\frac{|\sigma|}{|s|})\log 4.
\end{eqnarray*}

It is then easy to see that $F$ is bounded on the rays
$\hbox{arg}(s)=\theta, \pi-\theta, \pi+\theta, 2\pi-\theta$, where
$0<\theta<\pi/2$ with $\cos \theta=D$, since on these rays, $|\cos\theta|={|\sigma|\over |s|}=D$. Note that $f$ is of order $\le 1$ by the assumption,
a nonconstant $L$-function is of order $1$ (see e.g. \cite{Se} and \cite{St}), and the infinite product
in (\ref{function}) is also of order $\le 1$, which follows from (\ref{product}). Thus, $F$ must be of order at most $1$. We
then have that $F(s)=O\left(e^{|s|^{1+\epsilon}}\right)$ for any
$\epsilon>0$. Recall the Phragm\'en-Lindel\"of theorem (see e.g. \cite{T1}, p.177): Let $f$ be holomorphic in a sector between two straight lines making an angle of $\pi/\alpha$ at the origin and continuous on the boundary. If $|f(s)|\le M$ on the boundary and $f(s)=O(e^{r^{\beta}})$ as $r\to\infty$ uniformly in the sector, where $\beta<\alpha$, then $|f(s)|\le M$ in the entire sector. We see that $F$ satisfies the conditions of the theorem in each of the sectors bounded by the
above rays and thus $f$ is bounded in each of the sectors and thus in the entire complex plane. Therefore the entire function $F$ must
be a constant. But, $F$ has a zero at $s=0$ (see the choice of $n$). Thus $F$ and then $f-L$ must
be identically zero.

Next, we prove that the inequality in (\ref{G-notation}) is best possible. We will present a counterexample, in which $f$ is even not a Dirichlet series (and thus the theorem fails badly). To this end, consider
$$ L(s)=1+{2\over 4^s},  \quad    f(s)=(1+\frac{1}{s(1-s)})L(s).$$
Then it is easy to verify that
$$2^sL(s)=2^{1-s}\overline{L(1-{\overline s})},$$
which also clearly implies that
$$2^sf(s)=2^{1-s}\overline{f(1-{\overline s})}.$$
That is, both $f$ and $L$ satisfy the same functional equation. The zeros of $L(s)$ are $\frac{1}{\ln 4}(\ln 2+\pi i+2k\pi i)$, where $k$ is an integer, which readily implies that
$$\lim\limits_{r\to\infty}{n(r, Z(L))\over r}={\log 4 \over \pi}$$
and also that
$$\lim\limits_{r\to\infty}{n(r, Z(f))\over r}={\log 4 \over \pi}.$$
Now, take the exceptional set $G$ to be the entire set $Z(f)$. Then, $Z^+(f)\setminus G=\emptyset\subset Z^+(L)$. However, $f\not=L$. This proves the theorem.
\end{proof}

\medskip
\noindent{\bf Remark} {\bf (i)} It would be tempting to try to drop the condition of the order $\le 1$ for $f$ in Theorem 2.5. But, it is not the case. Consider
$$ L(s)=1+\frac{2}{4^s},  \quad    f(s)=\frac{1}{1+e^{s(1-s)}}L(s).$$
Then it is easy to check that $f$ is of order equal to $2$ with $\lim\limits_{\sigma\to +\infty}f(s)=1$. From the proof of Theorem 2.5, we see that both $f$ and $L$ satisfy the same functional equation. Note that $L$ and $f$ have the same zeros. Take the exceptional set $G$ to be the empty set. Then, $Z^+(f)\setminus G=Z^+(L)$. However, $f\not=L$.

\medskip
\noindent
{\bf (ii)} The condition that $\lim\limits_{\sigma\to +\infty}f(s)=1$ in the theorem cannot be dropped either. To see this, use the same function $ L(s)=1+{2\over 4^s}$ as in (i) but set $f(s)=\frac{1}{s(1-s)}L(s).$ Then $L$ and $f$ satisfy all the conditions of Theorem 2.5 with $G$ being the empty set, except that $\lim\limits_{\sigma\to +\infty}f(s)=0$. But,  $f\not=L$.

\bigskip
The above ideas may be carried over to L-functions. To demonstrate, we will do this specifically for the Dedekind zeta function of an algebraic number field, which encodes important arithmetic information of the field and has extensively been studied in number theory (see e.g. the monographs \cite{Ne} and \cite{HY}). Let $\kappa$ be a number field. Its Dedekind zeta function is defined by the Dirichlet series
$$\zeta_{\kappa}(s) =\sum\limits_{\mathfrak{a}}\frac{1}{\mathcal{N}(\mathfrak{a})^{s}}$$
for $\sigma>1$, where $\mathfrak{a}$ runs over the non-zero ideals of the ring $\kappa$ of integers of $\kappa$ and $\mathcal{N}(\mathfrak{a})$ denotes the absolute norm of $\mathfrak{a}$. It becomes the Riemann zeta function when the field is the rational numbers $\mathbb{Q}$.

The Dirichlet series converges absolutely for $\sigma>1$, it has an analytic continuation to a meromorphic function in $\mathbb{C}$ of order equal to $1$ with only a simple pole at $s = 1$. By the well-known Analytic Class Number Formula (see e.g. \cite{Ne}, p.467), the residue of $\zeta_{\kappa}$ at $s=1$ is given by
\begin{equation}\label{residue-formula}
\lim\limits_{s\to 1}(s-1)\zeta_{\kappa}=\frac{2^{r_1}(2\pi)^{r_2}c_{\kappa}R_{\kappa}}{w_{\kappa}{\sqrt{|D_{\kappa/\mathbb{Q}}|}}},
\end{equation}
where $r_1$ (resp. $r_2$) is the number of real (resp. complex) places of $\kappa$, $c_{\kappa}$ is the class number of $\kappa$, $R_{\kappa}$ is the regulator of $\kappa$, $D_{\kappa/\mathbb{Q}}$ is the discriminant of the field $\kappa$, and $w_{\kappa}$ denotes the number of roots of unity in $\kappa$. The function $\zeta_\kappa$ satisfies the
following functional equation (see e.g. \cite{Ne}, p.467)
\begin{equation}\label{kappa}
\zeta_\kappa(1-s)= |D_{\kappa/\mathbb{Q}}|^{s-\frac{1}{2}}
\left(\cos\frac{\pi s}{2}\right)^{r_1+r_2} \left(\sin\frac{\pi
s}{2}\right)^{r_2} \Gamma_{\mathbb{C}}(s)^n \zeta_\kappa(s),
\end{equation}
where $n=[\kappa:\mathbb{Q}]=r_1+2r_2$ and
$
\Gamma_{\mathbb{C}}(s)=2(2\pi)^{-s}\Gamma(s).
$

Let
$\Gamma_{\mathbb{R}}(s)=\pi^{-\frac{s}{2}}\Gamma\left(\frac{s}{2}\right).$ Then, the function
\begin{equation}\label{xi-a}
\xi_\kappa(s)=\frac{s}{2}(s-1)|D_{\kappa/\mathbb{Q}}|^{s/2}
\Gamma_{\mathbb{R}}(s)^{r_1}\Gamma_{\mathbb{C}}(s)^{r_2}\zeta_\kappa(s),
\end{equation}
is an entire function of order $1$ and satisfies the functional equation (see e.g.  \cite{Ne}, p.467)
\begin{equation}\label{xi-k}
\xi_\kappa(s)=\xi_\kappa(1-s).
\end{equation}

\begin{conjecture}[Grand Riemann Hypothesis for the Dedekind zeta function]
If $\zeta_\kappa(s)= 0$ and $0 \le {\rm Re}(s) \le 1$, then ${\rm Re}(s) = 1/2.$
\end{conjecture}

As before and for convenience of comparison, we still use $\rho_n$ to denote the zeros of $\zeta_{\kappa}$ satisfying $0\leq {\rm Re}(\rho_n)\leq 1$ and $s_\nu$ the zeros of $\zeta_{\kappa}$ on the half-line ${\rm Re}(s)=\frac{1}{2}$, ${\rm Im}(s)>0$, ordered with respect to increasing absolute values of their imaginary parts.

\medskip
We define
\begin{equation}\label{h-product}
h_\kappa(s)=\frac{2^{r_1+r_2-1}}{w_{\kappa}}c_{\kappa}R_{\kappa}\prod_{\nu=1}^\infty\left(1-\frac{s-s^2}{|s_\nu|^2}\right).
\end{equation}
Then with the same arguments for $h(s)$, we can show that the function $h_\kappa(s)$ possesses the following properties:

\medskip
(a) $h_\kappa$ is an entire function of order $\le 1$;

(b) The general factor $1-\frac{s-s^2}{|s_\nu|^2}$ in the infinite product has exactly the zeros at $s_v$ and $\overline{s_\nu}$ and thus the zeros of the function $h_\kappa$ are exactly $s_v$ and $\overline{s_\nu}$ (symmetrically with respect to the critical line),  $v=1, 2, \cdots$;

(c) $\lim\limits_{s\to 1}h_\kappa(s)=\frac{2^{r_1+r_2-1}}{w_{\kappa}}c_{\kappa}R_{\kappa}$;

(d) $h_\kappa$ satisfies the same equation (\ref{xi-k}) as $\xi_\kappa$ does, i.e.,
\begin{equation}\label{h-kappa}
h_\kappa(1-s)=h_\kappa(s).
\end{equation}

Further, we define $\eta_\kappa(s)$ by
\begin{equation}\label{eta-k}
h_\kappa(s)=\frac{s}{2}(s-1)|D_{\kappa/\mathbb{Q}}|^{s/2}
\Gamma_{\mathbb{R}}(s)^{r_1}\Gamma_{\mathbb{C}}(s)^{r_2}\eta_\kappa(s).
\end{equation}
Replacing $s$ by $1-s$, we then have that
\begin{equation*}
h_\kappa(1-s)=\frac{s}{2}(s-1)|D_{\kappa/\mathbb{Q}}|^{\frac{1-s}{2}}
\Gamma_{\mathbb{R}}(1-s)^{r_1}\Gamma_{\mathbb{C}}(1-s)^{r_2}\eta_\kappa(1-s),
\end{equation*}
which implies, in view of (\ref{eta-k}), that
\begin{eqnarray*}
& &\eta_{\kappa}(1-s)=\frac{h_{\kappa}(1-s)}{ \frac{s}{2}(s-1)|D_{\kappa/\mathbb{Q}}|^{ \frac{1-s}{2} }
\Gamma_{\mathbb{R}}(1-s)^{r_1}\Gamma_{\mathbb{C}}(1-s)^{r_2}\eta_\kappa(s)   }\\
& &=|D_{\kappa/\mathbb{Q}}|^{s-\frac{1}{2}}\frac{\Gamma_{\mathbb{R}}(s)^{r_1}\Gamma_{\mathbb{C}}(s)^{r_2}}{\Gamma_{\mathbb{R}}
(1-s)^{r_1}\Gamma_{\mathbb{C}}(1-s)^{r_2}}\eta_{\kappa}(s) \\
& &=|D_{\kappa/\mathbb{Q}}|^{s-\frac{1}{2}}\left(\cos\frac{\pi s}{2}\right)^{r_1+r_2} \left(\sin\frac{\pi
s}{2}\right)^{r_2}\Gamma_{\mathbb{C}}(s)^n \eta_\kappa(s)
\end{eqnarray*}
by virtue of the identity
\begin{eqnarray*}
& &\frac{\Gamma_{\mathbb{R}}(s)^{r_1}\Gamma_{\mathbb{C}}(s)^{r_2}}{\Gamma_{\mathbb{R}}
(1-s)^{r_1}\Gamma_{\mathbb{C}}(1-s)^{r_2}} \\
& &=\left(\cos\frac{\pi s}{2}\right)^{r_1+r_2} \left(\sin\frac{\pi
s}{2}\right)^{r_2} \Gamma_{\mathbb{C}}(s)^n,
\end{eqnarray*}
which can be directly verified or obtained from (\ref{xi-k}), (\ref{xi-a}) and (\ref{kappa}). Hence, $\eta_\kappa$ satisfies the same functional equation (\ref{kappa}) as $\zeta_{\kappa}$ does:
\begin{equation*}
\eta_\kappa(1-s)= |D_{\kappa/\mathbb{Q}}|^{s-\frac{1}{2}}
\left(\cos\frac{\pi s}{2}\right)^{r_1+r_2} \left(\sin\frac{\pi
s}{2}\right)^{r_2} \Gamma_{\mathbb{C}}(s)^n \eta_\kappa(s).
\end{equation*}

\bigskip
\begin{theorem}\label{GRH-eqthm}
The Grand Riemann hypothesis is true if and only if $\zeta_{\kappa}(s)=\eta_{\kappa}(s)$.
\end{theorem}

\begin{proof}
The sufficiency is clear since all the zeros of $\eta_{\kappa}$ on ${\rm Re}(s)\ge 0$ lie on the line  ${\rm Re}(s)=\frac{1}{2}$ by the construction of $\eta_{\kappa}(s)$ and $h_{\kappa}(s)$ (cf. Property (b) of $h_{\kappa}$).

For the necessity, if the Grand Riemann Hypothesis holds, then $\zeta_{\kappa}(s)$ and $\eta_{\kappa}(s)$ have the same zeros in the entire complex plan; thus we have that
$\eta_{\kappa}(s)=e^{as+b}\zeta_{\kappa}(s)$ for some complex numbers $a, b$, in view of the fact that $\zeta_{\kappa}$ is of order $1$ and $\eta_{\kappa}$ is of order $\le 1$ (cf. Property (a) of $h_{\kappa}$). By applying (\ref{eta-k}), (\ref{h-kappa}), (\ref{xi-k}), and (\ref{xi-a}), we deduce that
\begin{eqnarray*}
& &e^{as+b}=\frac{\eta_{\kappa}(s)}{\zeta_{\kappa}(s)}=\frac{h_{\kappa}(s)}{\xi_{\kappa}(s)}\\
& &=\frac{h_{\kappa}(1-s)}{\xi_{\kappa}(1-s)}=\frac{\eta_{\kappa}(1-s)}{\zeta_{\kappa}(1-s)}=e^{a(1-s)+b},
\end{eqnarray*}
which implies that $a=0$. Then $\eta_{\kappa}(s)=e^{b}\zeta_{\kappa}(s)$ and thus \begin{equation}\label{limit}
(s-1)\eta_{\kappa}(s)=e^{b}(s-1)\zeta_{\kappa}(s).
\end{equation}

Next, by the Analytic Class Number Formula (\ref{residue-formula}), the residue of $\zeta_{\kappa}$ at $s=1$ is
$$\lim\limits_{s\to 1}(s-1)\zeta_{\kappa}=\frac{2^{r_1}(2\pi)^{r_2}c_{\kappa}R_{\kappa}}{w_{\kappa}{\sqrt{|D_{\kappa/\mathbb{Q}}|}}}.$$
On the other hand, by  (\ref{eta-k}) and Property (c) of $h_{\kappa}$ we have that
\begin{eqnarray*}
& &\lim\limits_{s\to 1}(s-1)\eta_{\kappa}=\lim\limits_{s\to 1}\frac{(s-1)
{h_\kappa(s)}}{{\frac{s}{2}(s-1)|D_{\kappa/\mathbb{Q}}|^{s/2}
\Gamma_{\mathbb{R}}(s)^{r_1}\Gamma_{\mathbb{C}}(s)^{r_2}}}\\
& &=\frac{2^{r_1+r_2-1}c_{\kappa}R_{\kappa}}{\frac{1}{2}\pi^{-r_2}w_{\kappa}\sqrt{|D_{\kappa/\mathbb{Q}}|}}\\
& &=\frac{2^{r_1}(2\pi)^{r_2}c_{\kappa}R_{\kappa}}{w_{\kappa}{\sqrt{|D_{\kappa/\mathbb{Q}}|}}}.
\end{eqnarray*}

Taking the limit $s\to 1$ in (\ref{limit}), we deduce that $1=e^b$. This proves that $\eta_{\kappa}=\zeta_{\kappa}.$
\end{proof}

To conclude the paper, we give the analogue of Theorem  2.4 for the Dedekind zeta function. We can write $\zeta_{\kappa}$ as a normal Dirichlet series,
$\zeta_{\kappa}=\sum\limits_{n=1}^{\infty}\frac{a(n)}{n^s}$, where the coefficients $a(n)$ now represent the number of ideals of norm $n$ and $a(1)=1$, corresponding to the ideal $\mathcal{O}_{\kappa}$. If the Grand Riemann Hypothesis is true, then by the necessary condition of Theorem 2.7, $\zeta_{\kappa}=\eta_{\kappa}$. Thus,  $\lim\limits_{\sigma\to +\infty}\eta_\kappa(s)=\lim\limits_{\sigma\to +\infty}\zeta_\kappa(s)=1.$ Conversely, if  $\lim\limits_{\sigma\to +\infty}\eta_\kappa(s)=1$, then by Theorem 2.5 and in view of the fact that all the conditions there are satisfied with $f=\eta_{\kappa}$ and $L=\zeta_{\kappa}$, we have the uniqueness that $\zeta_{\kappa}=\eta_{\kappa}$ and thus the Grand Riemann Hypothesis holds by the sufficient condition of Theorem 2.7. This shows that we have the following

\begin{theorem}
The Grand Riemann Hypothesis is true if and only if
$$
    \lim_{\sigma\to +\infty}\eta_\kappa(s)=1.
$$
\end{theorem}


\begin{thebibliography}{99}
\addcontentsline{toc}{chapter}{Bibliography}
\bibitem{BC} S. Bochner and K. Chandrasekharan, On Riemann's functional equation, Ann. of Math. 63 (1956), 336-360.
\bibitem{CY} M. Cardwell and Z. Ye, A uniqueness theorem of L-functions with rational moving targets, J. Math
Anal. 5(2014), 16-19.
\bibitem{GHK} S.M. Gonek, J. Haan, and H. Ki, A uniqueness theorem for functions in the extended Selberg class, Math. Z. 278 (2014), 995-1004.
\bibitem{Ha} H. Hamburger, \"Uber die Riemannsche Funktionalgleichung der $\zeta$-Funktion, Math. Z. 10(1921), 240-254.
\bibitem{HLY} P.C. Hu, P. Li and C.C. Yang, \textit{Unicity of meromorphic mappings}, Kluwer Academic Publishers, 2003.
\bibitem{HY} P.C. Hu and C.C. Yang,  \textit{Distribution theory of algebraic numbers}, Walter de Gruyter, 2008.
\bibitem{Ki} H. Ki, A remark on the uniqueness of the Dirichlet series with a Riemann-type function equation, Advances in Math. 231(2012), 2484-2490.
\bibitem{KL} H. Ki and B.Q. Li, On uniqueness in the extended Selberg class of Dirichlet series, Proc. Amer. Math. Soc. 141 (2013), 4169-4173.
\bibitem{Kn} M.I. Knopp, On Dirichlet series satisfying Riemann's functional equation, Invent. Math. 372(1994), 361-372.
\bibitem{Li1} B.Q. Li,  On common zeros of L-functions, Math. Z. 272(2012), 1097-1102.
\bibitem{Li2} B.Q. Li, A uniqueness theorem for Dirichlet series satisfying a Riemann type functional equation, Advances in Math. 226(2011), 4198-4211.
\bibitem{Ne} J. Neukireh, \textit{Algebraic Number Theory}, Springer-Verlag, Berlin, 1999.
\bibitem{Se} Selberg, A., Old and new conjectures and results about a class of Dirichlet series, Proceedings of the Amalfi Conference on Analytic Number Theory (Maiori, 1989), E. Bombieri et al. (eds.), Univ. Salerno, Salerno, 1992,  367-385.
\bibitem{St} J. Steuding, \textit{Value distribution of $L$-functions}, Lecture Notes in Mathematics, 1877, Springer-Verlag, 2007.
\bibitem{T1} E.C. Titchmarsh,  \textit{The Theory of Functions}, Second Edition, Oxford University Press, 1968.
\bibitem{T2} E.C. Titchmarsh,  \textit{The Theory of the Riemann Zeta Function}, Second Edition, Oxford University Press, 1988.
\bibitem{YY} C.C. Yang and H.X. Yi, \textit{Uniqueness Theory of Meromorphic Functions}, Kluwer Academic Publishers, 2003.


\end{thebibliography}
\end{document}